\documentclass[11pt]{amsart}
\usepackage{amsmath, amsfonts, amssymb, color, enumitem}

\usepackage{graphicx}
\usepackage{geometry}
\usepackage[utf8]{inputenc}
\usepackage[T1]{fontenc}
\usepackage{xcolor}
\usepackage{hyperref}

\usepackage{ulem}

\usepackage{tikz}
\usetikzlibrary{arrows}
\usetikzlibrary{patterns}

\usepackage{dsfont}

\definecolor{boubelcolor}{rgb}{.65,0.05,0}

\DeclareUnicodeCharacter{2260}{\colorlet{memoireboubel}{.}\color{boubelcolor}} 
\DeclareUnicodeCharacter{00B1}{\color{memoireboubel}{}} 

\DeclareUnicodeCharacter{00A3}{\colorlet{memoirejuillet}{.}\color{blue}} 
\DeclareUnicodeCharacter{00A7}{\color{memoirejuillet}{}} 

\newtheorem{theorem}{Theorem}[section]
\newtheorem{lemma}[theorem]{Lemma}
\newtheorem{proposition}[theorem]{Proposition}

\newtheorem{remark}[theorem]{Remark}

\numberwithin{equation}{section}

\newcommand{\de}{\, \mathrm{d}}

\newcommand{\leqc}{\leq_{c}}

\newcommand{\Id}{\mathsf{Id}}

\newcommand{\R}{\mathbb{R}}

\begin{document}
\title[Instability of Martingale optimal transport]{Instability of Martingale optimal transport \\ in dimension $\boldsymbol{d \geq 2}$}
\author{Martin Br\"uckerhoff \address[Martin Br\"uckerhoff]{Universit\"at M\"unster, Germany} \email{martin.brueckerhoff@uni-muenster.de} \hspace*{0.5cm} Nicolas Juillet \address[Nicolas Juillet]{Université de Strasbourg et CNRS, France} \email{nicolas.juillet@math.unistra.fr} } 
\thanks{MB is funded by the Deutsche Forschungsgemeinschaft (DFG, German Research Foundation) under Germany's Excellence Strategy EXC 2044 –390685587, Mathematics Münster: Dynamics–Geometry–Structure. }

\begin{abstract}
Stability of the value function and the set of minimizers w.r.t.\  the given data is a desirable feature of optimal transport problems. For the classical Kantorovich transport problem, stability is satisfied under mild assumptions and in general frameworks such as the one of Polish spaces. However, for the martingale transport problem several  works based on different strategies established stability results for $\R$ only. We show that the restriction to dimension $d=1$ is not accidental by presenting a sequence of marginal distributions on $\mathbb{R}^2$ for which the martingale optimal transport problem is neither stable w.r.t.\ the value nor the set of minimizers. Our construction adapts to any dimension $d \geq 2$. For $d\geq 2$ it also provides a contradiction to the martingale Wasserstein inequality established by Jourdain and Margheriti in $d=1$.

	\smallskip
	
	\normalem

\end{abstract}

\date{\today}
\maketitle

\section{Introduction}

For two probability measures $\mu$ and $\nu$ on $\mathbb{R}^d$ let $\Pi(\mu,\nu)$ denote the set of all couplings between $\mu$ and $\nu$, i.e.\ the set of all probability measures on $\mathbb{R}^d \times \mathbb{R}^d$ which have marginal distributions $\mu$ and $\nu$. Let $c:\R^d\times \R^d\to \R$ be measurable and integrable with respect to the elements of $\Pi(\mu,\nu)$. The classical optimal transport problem is given by 
\begin{equation} \label{eq:OT} \tag{OT}
	V_c(\mu,\nu) = \inf_{\pi \in \Pi(\mu,\nu)} \int _{\mathbb{R}^d \times \mathbb{R}^d} c(x,y) \de \pi(x,y).
\end{equation}
For the cost function $c_1(x,y) := \Vert y-x\Vert$ (where $\Vert \cdot \Vert$ is the Euclidean norm) the $1$-Wasserstein distance $\mathcal{W}_1 := V_{c_1}$ is a metric on $\mathcal{P}_1(\mathbb{R}^d)$, the space of probability measures $\mu$ that satisfy $\int_{\R^d} \|x\| d\mu(x)<\infty$.

Two probability measures $\mu, \nu \in \mathcal{P}_1(\mathbb{R}^d)$ are said to be in convex order, denoted by $\mu \leqc \nu$, if $\int _{\mathbb{R}^d} \varphi \de \mu \leq \int _{\mathbb{R}^d} \varphi \de \nu$ for all convex functions $\varphi \in L^1(\nu)$. If $\mu \leqc \nu$, Strassen's theorem yields that there exists at least one martingale coupling between $\mu$ and $\nu$. A martingale coupling is a coupling $\pi \in \Pi(\mu,\nu)$ for which there exists a disintegration $(\pi_x)_{x \in \mathbb{R}^d}$ such that 
\begin{equation} \label{eq:MartDis}
	\int _{\mathbb{R}} y  \de \pi _x (y) = x \quad \text{for } \mu\text{-a.e. } x \in \mathbb{R}^d. 
\end{equation}
If $\mu \leqc \nu$,  the martingale optimal transport problem is given by
\begin{equation} \label{eq:MOT} \tag{MOT}
	V^M_c(\mu,\nu) = \inf_{\pi \in \Pi_M(\mu,\nu)} \int _{\mathbb{R}^d \times \mathbb{R}^d} c(x, y) \de \pi(x,y)
\end{equation}
where $\Pi_M(\mu,\nu)$ denotes the set of all martingale couplings between $\mu$ and $\nu$.

\subsection*{Stability  in d $\boldsymbol{=1}$}
Let us recall two reasons why stability results are crucial from an applied perspective. Firstly, they enable the strategy of approximating the problem by a discretized problem or by any other that can rapidly be solved computationally (cf.\ \cite{AlCoJo20,JoPa20}). Secondly, any application to noisy data would require stability for the results to be meaningful.
In relation with \eqref{eq:MOT}, this discussion is motivated by its connection to (robust) mathematical finance (cf.\  \cite{BeHePe13,He17}).

Let $\mu, \nu \in \mathcal{P}_1(\mathbb{R})$ with $\mu \leqc \nu$, and $(\mu_n)_{n \in \mathbb{N}}$ and $(\nu_n)_{n \in \mathbb{N}}$ be sequences of probability measures on $\mathbb{R}$ with finite first moment such that $\lim _{n \rightarrow \infty}\mathcal{W}_1(\mu_n, \mu) = 0,$ $\lim_{n \rightarrow \infty} \mathcal{W}_1(\nu_n,\nu)=0$ and $\mu _n \leqc \nu_n$ for all $n \in \mathbb{N}$. The following stability results are available:
\begin{itemize}
	
	\item [(S1)] \textit{Accumulation Points of Minimizers:}
	Let $c$ be a continuous cost function which is sufficiently integrable (e.g.\ $|c(x,y)|\leq A(1+|x|+|y|)$)  and let $\pi_n$ be a minimizer of the \eqref{eq:MOT} problem between $\mu_n$ and $\nu_n$ for all $n \in \mathbb{N}$. Any weak accumulation point of $(\pi_n)_{n \in \mathbb{N}}$ is a minimizer of \eqref{eq:MOT} between $\mu$ and $\nu$.
	
	\item [(S2)] \textit{Continuity of the Value:} Let $c$ be a continuous cost function which is sufficiently integrable (e.g.\ $|c(x,y)|\leq A(1+|x|+|y|)$).
	There holds
	\begin{equation*} 
		\lim_{n \rightarrow \infty} V^M_c(\mu_n,\nu_n) = V^M_c(\mu,\nu).
	\end{equation*}

	\item [(S3)] \textit{Approximation:} For all $\pi \in \Pi_M(\mu,\nu)$ there exists a sequence $(\pi _n)_{n \in \mathbb{N}}$ of martingale couplings between $\mu_n$ and $\nu_n$ that converges weakly to $\pi$.
\end{itemize}

This constitutes the heart of the theory of stability recently consolidated for the martingale transport problem on the real line. Before we go more into the details of the literature let us stress that with (S3) any minimizer can be approximated by a sequence of martingale transport with prescribed marginals. Therefore, under mild assumptions (S3) implies (S2).
Moreover, due to the tightness of $\bigcup_{n \in \mathbb{N}} \Pi_M(\mu_n,\nu_n)$, (S2) implies (S1). 

Early versions of (S1) and (S2) for special classes of cost-functions were obtained by Juillet \cite{Ju14} and later by Guo and Obloj \cite{GuOb17}.  The general version of (S1) and (S2) was first shown by Backhoff-Veraguas and Pammer \cite[Theorem 1.1, Corollary 1.2]{BaPa19} and Wiesel \cite[Theorem 2.9]{Wi20}. Only very recently, Beiglböck, Jourdain, Margheriti and Pammer \cite{BeJoMaPa20} have proven (S3). 
We want to stress that (S1), (S2) and (S3) are given in a minimal formulation and that in the articles some aspects of the results are notably stronger. For instance, the cost function $c$ in (S1) and (S2) can be replaced by a uniformly converging sequence $(c_n)_{n\in \mathbb N}$ \cite{BaPa19}. Moreover, it is an important achievement that on top of weak convergence we have convergence w.r.t.\  (an extension of) the adapted Wasserstein metric for the approximation in (S3) \cite{BeJoMaPa20} and for the convergence in (S1) \cite{BeJoMaPa20b}, see also \cite{Wi20}.  Finally, these stability results also hold for weak martingale optimal transport which is an extension of \eqref{eq:MOT} w.r.t.\ the structure of the cost function (cf.\ \cite[Theorem 2.6]{BeJoMaPa20b}). For further details we invite the interested reader to directly consult the articles.

The martingale Wasserstein inequality introduced by Jourdain and Margheriti  in \cite[Theorem 2.12]{JoMa20} belongs also to the context of stability and approximation and it appears for example as the important last step in the proof of (S3) in \cite{BeJoMaPa20}. In dimension $d= 1$ there exists a constant $C > 0$ independent of $\mu$ and $\nu$ such that 
\begin{equation} \label{eq:StabilityEq}\tag{MWI}
	\mathcal{M}_1(\mu,\nu) \leq C \mathcal{W}_1(\mu,\nu).
\end{equation}
where $\mathcal{M}_1(\mu,\nu)$ is the value of the \eqref{eq:MOT} problem between $\mu$ and $\nu$ w.r.t.\ the cost function $\Vert x-y \Vert$.
Moreover, they proved that $C = 2$ is sharp. 
For their proof Jourdain and Margheriti introduce a family of martingale couplings $\pi \in\Pi_M(\mu,\nu)$ that satisfy $\int _{\mathbb{R} \times \mathbb{R}} |x - y| \de \pi(x,y)\leq 2\mathcal{W}_1(\mu,\nu)$ (including the particularly notable inverse transform martingale coupling).


\subsection*{Instability in d $\boldsymbol{\geq 2}$}

The stability of \eqref{eq:OT} (and its extension to weak optimal transport \cite[Theorem 2.5]{BeJoMaPa20b}) is independent of the dimension. However, Beiglböck {\it et al.} had to restrict their stability theorem for (weak) \eqref{eq:MOT} in the critical step to dimension $d = 1$ (cf.\ \cite[Theorem 2.6 (b')]{BeJoMaPa20b}). Similarly, in dimension $d \geq 2$, Jourdain and Margheriti could only extend the martingale Wasserstein inequality  for product measures and for measures in relation through a homothetic transformation, see \cite[Section 3]{JoMa20}.  The difficulties in expanding these stability results to higher dimensions are not of technical nature but a consequence of instability of \eqref{eq:MOT} in higher dimensions without further assumptions.

In the following we construct a sequence of probability measures on $\mathbb{R}^2$ for which (S1), (S2) and (S3) do not hold. Moreover, we provide an example that shows that the inequality \eqref{eq:StabilityEq} does not hold in dimension $d = 2$ for any fixed constant $C > 0$ without further assumptions. Since one can  embed this example into $\mathbb{R}^d$ for any $d \geq 3$ by the map $\iota : (x,y) \mapsto (x,y,0,...,0)$, these results also fail in any higher dimension.

We denote by $P_\theta$ the one step probability kernel of the simple random walk along the line $l_\theta$ that makes an angle $\theta \in \left[0, \frac{\pi}{2}\right]$ with the $x$-axis. More precisely:
\begin{equation*}
	P_\theta :  \R^2  \ni (x_1,x_2)\mapsto \frac{1}{2} \left( \delta _{(x_1,x_2) + (\cos(\theta),\sin(\theta))} + \delta _{(x_1,x_2) - (\cos(\theta),\sin(\theta))}\right)\in \mathcal{P}_1(\R^2).
\end{equation*}
For $m,n \in \mathbb{N}_{\geq 1}$ we define  two probability measures on $\mathbb{R}^2$ by 
\begin{align*}
	 \mu _m := \sum _{i = 1} ^{m} \frac{1}{m} \delta_{(i,0)} \quad \text{and} \quad
	 \nu _{m,n} := \mu_m P_{\frac{\pi}{2n}}
\end{align*}
where $\mu_m P_{\frac{\pi}{2n}}$ denotes the application of the kernel $P_{\frac{\pi}{2n}}$ to $\mu _m$. Figure \ref{fig:Sketch} illustrates  $(\mu_2,\nu_{2,2})$ and $(\mu_3, \nu _{3,3})$.

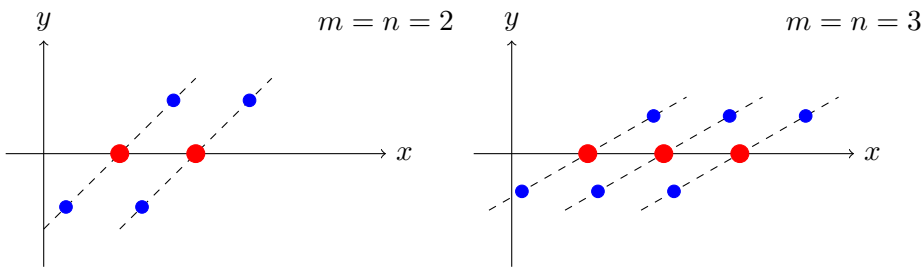
\begin{figure} \label{fig:Sketch}
	\begin{center}
		
		\begin{tikzpicture}
		\draw[->] (0,-1.5) -- (0,1.5);
		\node[above] at (0,1.5) {$y$};
		\draw[->] (-0.5,0) -- (4.5,0);
		\node[right] at (4.5,0) {$x$};
		\node[above] at (4.5,1.5) {$m=n=2$};

		\filldraw[red] (1,0) circle (0.5ex);
		\filldraw (2,0) circle (0.5ex);
		
		\draw[dashed] (0,-1) -- (2,1);
		\draw[dashed] (1,-1) -- (3,1);
		
		\filldraw[red] (1,0) circle (.71ex);
		\filldraw[red] (2,0) circle (0.71ex);
		\filldraw[blue] (1-0.707,-0.707)	 circle (0.5ex);
		\filldraw[blue] (1-0.707+1,-0.707) circle (0.5ex);
		\filldraw[blue] (1+0.707,0.707) circle (0.5ex);
		\filldraw[blue] (1+0.707+1,0.707) circle (0.5ex);
		
		\end{tikzpicture}
		\begin{tikzpicture}
		\draw[->] (0,-1.5) -- (0,1.5);
		\node[above] at (0,1.5) {$y$};
		\draw[->] (-0.5,0) -- (4.5,0);
		\node[right] at (4.5,0) {$x$};
		\node[above] at (4.5,1.5) {$m=n=3$};
		
		\draw[dashed] (1-1.299,-0.75) -- (1+1.299,0.75);
		\draw[dashed] (2-1.299,-0.75) -- (2+1.299,0.75);
		\draw[dashed] (3-1.299,-0.75) -- (3+1.299,0.75);

		\filldraw[red] (1,0) circle (0.71ex);
		\filldraw[red] (2,0) circle (0.71ex);
		\filldraw[red] (3,0) circle (0.71ex);
		
		\filldraw[blue] (1-0.866,-0.5)	 circle (0.5ex);
		\filldraw[blue] (2-0.866,-0.5) circle (0.5ex);
		\filldraw[blue] (3-0.866,-0.5) circle (0.5ex);
		
		\filldraw[blue] (1+0.866,0.5)	 circle (0.5ex);
		\filldraw[blue] (2+0.866,0.5) circle (0.5ex);
		\filldraw[blue] (3+0.866,0.5) circle (0.5ex);
			
		\end{tikzpicture}

		\caption{The construction for $m=n =2$ and $m=n = 3$. The red circles indicate the Dirac measures of $\mu_m$ each with mass $\frac{1}{m}$ and the blue circles indicate the Dirac measures of $\nu_{m,n}$ each with mass $\frac{1}{2m}$.}
	\end{center}
\end{figure}

Since for any convex function $\varphi : \mathbb{R}^2 \rightarrow \mathbb{R}$ Jensen's inequality yields
\begin{equation*}
	\int_{\mathbb{R}^2} \varphi \de \nu_{m,n} = \int_{\mathbb{R}^2} \left( \int_{\mathbb{R}^2} \varphi \de P_{\frac{\pi}{2n}} (x, \cdot) \right) \de \mu_m (x) \geq \int_{\mathbb{R}^2} \varphi \de \mu_{m}, 
\end{equation*}
we have $\mu_m \leqc \nu_{m,n}$ for all $m,n \in \mathbb{N}_{\geq 1}$. 

\begin{lemma} \label{lemma:MartCpl}
	The martingale coupling $\pi_{m,n} := \mu_m(\Id, P_{\frac{\pi}{2n}})$ is the only martingale coupling between $\mu_m$ and $\nu_{m,n}$ for all $m,n \in \mathbb{N}_{\geq 1}$.
\end{lemma}


The sequence $(\mu_3, \nu_{3,n})_{n \in \mathbb{N}}$ serves as a counterexample to analogue versions of (S1), (S2) and (S3) in dimension $d = 2$.  The crucial observation is that whereas $\Pi_M(\mu_3, \nu_{3,n})$ consists of exactly one element for all $n \in \mathbb{N}$ by Lemma \ref{lemma:MartCpl}, there are infinitely many different martingale couplings between $\mu_3$ and the limit of $(\nu_{3,n})_{n \in \mathbb{N}}$.

\begin{proposition} \label{prop:NoStability}
	There holds $\lim _{n \rightarrow \infty} \mathcal{W}_1(\nu_{3,n},\mu_3P_0) = 0$.
	
	Moreover, we have the following:
	\begin{itemize}
		\item[(i)]  Let $c_1(x,y) := \Vert y-x\Vert$ for all $x,y \in \mathbb{R}^2$ and $\pi_n := \mu_3(\Id,P_{\frac{\pi}{2n}})$ for all $n \in \mathbb{N}_{\geq 1}$. The martingale couplings $\pi_n$ are minimizers of the \eqref{eq:MOT} problem between $\mu_3$ and $\nu_{3,n}$ w.r.t.\ $c_1$. Moreover, $(\pi_n)_{n \in \mathbb{N}}$ is weakly convergent but the limit is not an optimizer of \eqref{eq:MOT} w.r.t.\ $c_1$ between (its marginals) $\mu_3$ and $\mu_3P_0$.
		\item[(ii)] Let $c_1$ be defined as in (i). There holds 	
		\begin{equation*}
			\lim _{n \rightarrow \infty} V_{c_1}^M(\mu_3, \nu_{3,n}) = 1 > V_{c_1}^M (\mu_3, \mu_3P_0).
		\end{equation*}
		\item[(iii)] The set $\Pi_M(\mu_3,\mu_3P_0) \setminus \{\mu_3(\Id,P_0)\} $ is non empty and no element in this set can be approximated by a weakly convergent sequence $(\pi_n)_{n \in \mathbb{N}}$ of martingale couplings $\pi _n \in \Pi_M(\mu_3, \nu_{3,n})$.
	\end{itemize}
\end{proposition}

The sequence $(\mu_n, \nu_{n,n})_{n \in \mathbb{N}}$ shows that there cannot exist a constant $C >0$ for which the inequality \eqref{eq:StabilityEq} holds in dimension $d = 2$. 

\begin{proposition} \label{prop:NoConst}
	There holds 
	\begin{equation*}
	\lim _{n \rightarrow \infty} \frac{\mathcal{M}_1(\mu_n,\nu_{n,n})}{\mathcal{W}_1(\mu_n,\nu_{n,n})} = + \infty.
	\end{equation*}
\end{proposition}

\begin{remark} 
The theory of \ref{eq:MOT} in dimension two and further is also challenging in other aspects. For instance, the concept of irreducible components and convex paving can be directly reduced to the study  of potential functions in dimension $d = 1$, whereas there are at least three different advanced approaches in dimension $d \geq 2$  (cf.\ \cite{GhKiLi19,MaTo19,ObSi17}).
On the level of processes we would like to remind the reader that a higher dimensional version of Kellerer's theorem is still not proved or disproved. One major obstacle is that the one-dimensional proof via Lipschitz-Markov kernels cannot be extrapolated, see \cite[Section 2.2]{Ju16} where a method similar to ours is used.
\end{remark}

\section{Proofs} \label{sec:proofs}

We denote by $f_{\#} \mu$ the push-forward of the measure $\mu$ under the function $f$.

\begin{proof}[Proof of Lemma \ref{lemma:MartCpl}]
	Let $m,n\geq 2$ be integers and $\theta_n := \frac{\pi}{2n}$. We denote by $L_{\theta_n}$ the projection parallel to the line $l_{\theta_n} = \{(x_1, \tan(\theta_n)x_1) : x_1 \in \mathbb{R}\}$ onto the $x$-axis, i.e.\
	\begin{equation*}
		L_{\theta_n} : \mathbb{R}^2 \ni (x_1,x_2) \mapsto x_1 - \tan(\theta_n)^{-1}x_2\in\mathbb{R}.
	\end{equation*}
	Moreover, by setting $\tilde{\nu} := (L_{\theta_n})_{\#} \nu_{m,n}$ and $\tilde{\mu} := (L_{\theta_n})_{\#} \mu_{m}$ one has
	\begin{equation}  \label{eq:SameMarg}
		\tilde{\nu} = \frac{1}{m} \sum_{i = 1} ^m \delta_{i} = \tilde{\mu}.
	\end{equation}
	
	Let $\pi \in \Pi_M(\mu_{m},\nu_{m,n})$. As $L_{\theta_n}$ is a linear map, $\tilde{\pi} := (L_{\theta_n}\otimes L_{\theta_n})_{\#} \pi$ is a martingale coupling of $\tilde{\mu}$ and $\tilde{\nu}$. Indeed, for all $\varphi \in C_b(\mathbb{R}^2)$ there holds
	\begin{equation*}
		\int _{\mathbb{R}^2} \varphi(x)(y-x) \de \tilde \pi(x) = L_{\theta_n} \left(\int _{\mathbb{R}^2} \varphi(L_{\theta_n}(x))(y-x) \de \pi(x) \right) = 0
	\end{equation*}
	and this property is equivalent to $\tilde{\pi}$ being a martingale coupling.
	Jensen's inequality in conjunction with \eqref{eq:SameMarg} yields 
	\begin{equation*}
		\int _{\mathbb{R}} x^2 \de \tilde{\mu}(x) \leq \int _{\mathbb{R}} \left( \int _{\mathbb{R}} y^2 \de \tilde{\pi}_x(y) \right) \de \tilde{\mu}(x) =  \int _{\mathbb{R}} y^2 \de \tilde{\nu}(y) = \int _{\mathbb{R}} x^2 \de \tilde{\mu}(x)
	\end{equation*}
	where $(\tilde{\pi}_x)_{x \in \mathbb{R}}$ is a disintegration of $\tilde{\pi}$ that satisfies \eqref{eq:MartDis}. Since the square is a strictly convex function, there holds  $\int _{\mathbb{R}} y^2 \de \tilde{\pi}_x=x^2$ if and only if $\tilde{\pi}_x=\delta_x$. Thus, $\tilde{\pi} = \tilde{\mu}(\Id,\Id)$ and we obtain
	\begin{equation*}
		x_1 = L_{\theta_n}(x_1,x_2) = L_{\theta_n}(y_1,y_2) \quad \text{for } \pi \text{-a.e. } ((x_1,x_2),(y_1,y_2)) \in \mathbb{R}^2 \times \mathbb{R}^2
	\end{equation*}
	because $\mathrm{supp}(\mu_m) \subset \mathbb{R} \times \{0\}$. Hence, the martingale transport plan $\pi$ is only transporting along the lines parallel to $l_{\theta_n}$. Since there are exactly two points in the support of $\nu_{m,n}$ that lie on the same line, and we are looking for a martingale coupling, we have
	\begin{equation*}
		\pi = \mu_m(\Id, P_{\theta_n}).\qedhere
	\end{equation*}
\end{proof}

\begin{lemma} \label{lemma:W1Dist}
	For all $m \in \mathbb{N}\setminus\{0\}$ and $\theta \in \left[0, \frac{\pi}{2} \right]$ one has
	\begin{equation*}
		\mathcal{W}_1(\mu_m P_0, \mu_m P_{\theta}) \leq \theta.
	\end{equation*}
\end{lemma}

\begin{proof} The inequality consists merely of a comparison of angle and chord. Alternatively, for all $m \in \mathbb{N}$ and $\theta \in \left[0, \frac{\pi}{2} \right]$ we directly compute
	\begin{align*}
		\mathcal{W}_1(\mu_m P_0, \mu_m P_{\theta}) &\leq \int _{\mathbb{R}^2} \mathcal{W}_1(\delta_x P_0, \delta _x P_{\theta}) \de \mu_m(x) \\
		&= \left\vert \left(\sin (\theta), \cos (\theta) -1 \right) \right\vert = \sqrt{2(1 - \cos(\theta))}=2\sin(\theta/2) \leq \theta.\qedhere
	\end{align*}
\end{proof}

\begin{proof}[Proof of Proposition \ref{prop:NoStability}]
	By Lemma \ref{lemma:W1Dist}, we know
	\begin{equation*}
		\lim _{n \rightarrow \infty}\mathcal{W}_1(\nu_{3,n},\mu_3P_0) = \lim _{n \rightarrow \infty}\mathcal{W}_1(\mu_3 P_{\frac{\pi}{2n}},\mu_3P_0) = 0.
	\end{equation*}
	Moreover, for all $n \in \mathbb{N}$ Lemma \ref{lemma:MartCpl} yields that $\pi_n := \mu_3(\Id, P_{\frac{\pi}{2n}})$ is the only martingale coupling between $\mu_3$ and $\nu_{3,n}$ and therefore automatically the unique minimizer of the \eqref{eq:MOT} problem between these two marginals w.r.t.\ to any cost function. The sequence $(\pi_n)_{n \in \mathbb{N}}$ converges weakly to $\pi := \mu_3(\Id,P_0) \in \Pi_M(\mu_3,\mu_3P_0)$. Note that 
	\begin{align*}
		\pi' :=& \frac{1}{6} \left( \delta_{((1,0),(1,0))} + 2 \delta_{((2,0),(2,0))} + \delta_{((3,0),(3,0))}\right) \\ &+ \frac{1}{24} \left(3\delta_{((1,0),(0,0))} +  \delta_{((1,0),(4,0))} + \delta_{((3,0),(0,0))} + 3 \delta_{((3,0),(4,0))} \right)
	\end{align*}
	is a martingale coupling between $\mu_3$ and $\mu_{3}P_0=\mu_3-\frac13\left(\frac{\delta_1+\delta_3}2-\frac{\delta_0+\delta_4}2\right)$ different from $\pi$ where only the mass not shared by $\mu_3$ and $\mu_{3}P_0$ is moved.
	
	Item (i): Since $\pi$ is the weak limit of the sequence $(\pi_n)_{n \in \mathbb{N}}$, it is the only accumulation point. But as we see below in (ii), $\pi$ is not the minimizer of the \eqref{eq:MOT} problem between $\mu_3$ and $\mu_3P_0$ w.r.t.\ $c_1$.
	
	Item (ii): There holds
	\begin{align*}
		\lim _{n \rightarrow \infty} V^M_{c_1}(\mu_3, \nu_{3,n}) &=\lim _{n \rightarrow \infty} \int_{\mathbb{R}^2 \times \mathbb{R}^2} \Vert x-y\Vert \de \pi_n = 1 \\ &> \frac{1}{2} = \int_{\mathbb{R}^2 \times \mathbb{R}^2} \Vert x-y\Vert \de \pi' \geq V_{c_1}^M(\mu_3,\mu_3P_0).
	\end{align*}	
	In fact, according to Lim's result \cite[Theorem 2.4]{Li20},  under an optimal martingale transport w.r.t.\ $c_1$  the shared mass between the marginal distribution is not moving. Since $\pi'$ is the unique martingale transport  between $\mu_3$ and $\mu_3P_0$ with this property, it is the minimizer of this \eqref{eq:MOT} problem and $V_{c_1}^M(\mu_3,\mu_3P_0) = \frac{1}{2}$.

	Item (iii): Since $\pi$ is the weak limit of the solitary elements of $\Pi_M(\mu_3, \nu_{3,n})$, no element of $\Pi_M(\mu_3,\mu_3P_0) \setminus \{\mu_3(\Id,P_0)\} $ can be approximated and $\pi'$ is an element of this set.
\end{proof}

\begin{proof}[Proof of Proposition \ref{prop:NoConst}]
	Let $n \in \mathbb{N}$.
	By Lemma \ref{lemma:MartCpl}, $\mu_n(\Id,P_{\frac{\pi}{2n}})$ is the only martingale coupling between $\mu_n$ and $\nu_{n,n}$. Thus,
	\begin{equation*}
		\mathcal{M}_1(\mu_n,\nu_{n,n}) = 1.
	\end{equation*}
	Since $\mathcal{W}_1$ is a metric on $\mathcal{P}_1(\mathbb{R}^2)$, the triangle inequality yields
	\begin{align*}
		\mathcal{W}_1(\mu_n, \nu_{n,n}) \leq \mathcal{W}_1(\mu_n, \mu_n P^0) + \mathcal{W}_1(\mu_n P^0, \nu _{n,n}).
	\end{align*}
	We can easily compute
	\begin{equation*}
		\mu_n P^0 = \frac{1}{2n} \left( \sum _{i = 1} ^n \delta_{(i-1,0)} + \sum _{i = 1} ^n \delta_{(i+1,0)} \right)
	\end{equation*}
	and therefore $\mathcal{W}_1(\mu_n, \mu_n P^0) = \frac{1}{n}$. By Lemma \ref{lemma:W1Dist}, there holds $\mathcal{W}_1(\mu_n P^0, \nu _{n,n}) \leq \frac{\pi}{2n}$.
	Hence, we obtain
	\begin{equation*}
		\lim_{n \rightarrow \infty} \frac{\mathcal{M}_1(\mu_n,\nu_{n,n})}{\mathcal{W}_1(\mu_n,\nu_{n,n})} \geq \lim_{n \rightarrow \infty} \frac{1}{\frac{1}{n} + \frac{\pi}{2n}} = + \infty.
	\end{equation*}
\end{proof}

\begin{remark}[Variations] \label{rem:Variants}
Our construction may appear somewhat degenerate since $\mu_m$ is  discrete and supported on a lower dimensional subspace of $\mathbb{R}^2$. However, it is not particularly difficult to adapt the present construction with new measures that appear more general but yield the same result:
\begin{itemize}
\item[(ii)] One could replace the rows of Dirac measures by uniform measures on parallelograms. More precisely, we could set 
\begin{equation*}
\tilde{\mu}_{m,n}:= \mathrm{Unif}_{F_{m,n}} \quad \text{and} \quad \tilde{\nu} _{m,n} := \frac{1}{2} \left( \mathrm{Unif}_{F_{m,n}^+} + \mathrm{Unif}_{F_{m,n}^-} \right)
\end{equation*}
where  $F_m$ denotes the parallelogram spanned by the points 
\begin{equation*}
	-v_n , \quad -v_n + (m,0) , \quad v_n + (m,0) \quad \text{and} \quad v_n 
\end{equation*}
with $v_n:= \frac{1}{3}\left(\cos\left(\frac{\pi}{2n}\right), \sin\left(\frac{\pi}{2n}\right) \right) \in \mathbb{R}^2$ 
and $F_{m,n}^{\pm}$ is the translation of this parallelogram by $\pm 3v_n$ (cf.\ Figure \ref{fig:Sketch2}). By the same argument as in Lemma \ref{lemma:MartCpl}, any martingale coupling between $\tilde{\mu}_{m,n}$ and $\tilde{\nu}_{m,n}$ can only transport along lines parallel to $\{(x, \tan\left(\frac{\pi}{2n}\right)x) : x \in \mathbb{R}\}$. In contrast to the situation in Lemma \ref{lemma:MartCpl}, the martingale transport along one of these parallel lines is no longer unique but 
every $\pi \in \Pi_{M}(\tilde{\mu}_{m,n}, \tilde{\nu}_{m,n})$ satisfies $\pi\left(|x-y| < \frac{1}{3}\right) = 0$  for all $m,n \in \mathbb{N}$ because the supports are disjoint. This restriction carries over to any weak accumulation point of those martingale couplings and is sufficient to show analogous versions of Proposition \ref{prop:NoStability} and Proposition \ref{prop:NoConst}.
\item [(iii)] One could replace $\mu_m$ and $\nu_{m,n}$ by 
\begin{equation*}
	\tilde{\mu}_m := (1-\epsilon)\mu_m+\epsilon\gamma \quad \text{and} \quad \tilde{\nu} _{m,n} := (1-\epsilon)\nu_{m,n}+\epsilon\gamma
\end{equation*}
where $\varepsilon \in (0,1)$ and $\gamma$ is a probability measure with full support (e.g.\ a standard normal distribution). There holds $\mathcal{W}_1(\tilde{\mu}_m,\tilde{\nu}_{m,n}) = (1-\epsilon)\mathcal{W}_1(\mu_m,\nu_{m,n})$ since $\mathcal{W}_1$ derives of the Kantorovich-Rubinstein norm \cite{KaRu58} (alternatively see   \cite[Bib.\ Notes of Ch.6 ]{Vi09} or \cite[\S *11.8]{Du02}) and $\mathcal{M}_1(\tilde{\mu}_m,\tilde{\nu}_{m,n}) = (1-\epsilon)\mathcal{M}_1(\mu_m,\nu_{m,n})$ by a result of Lim  \cite[Theorem 2.4]{Li20}.
\end{itemize}  
\end{remark}
\begin{remark}
Finally, we would like to point out that Propositions \ref{prop:NoStability} and \ref{prop:NoConst} and their proofs are not depending on the choice of the Euclidean norm while defining $c_1$.
\end{remark}

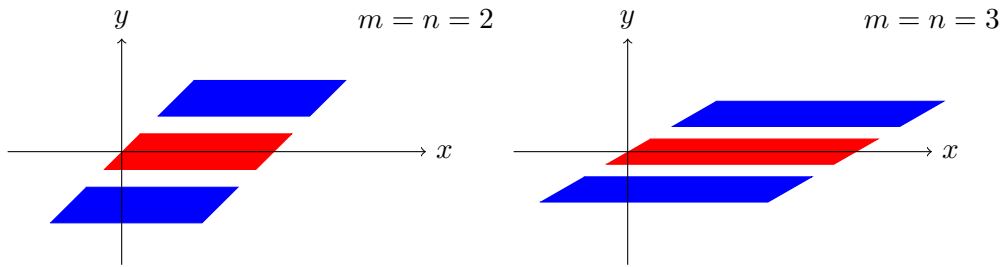
\begin{figure} \label{fig:Sketch2}
	\begin{center}
		
		\begin{tikzpicture}

		\filldraw[red] (-0.236,-0.236) -- (-0.236+2,-0.236) -- (0.236+2,0.236) -- (0.236,0.236);
		
		\filldraw[blue] (-0.236+3*0.236,-0.236+3*0.236) -- (-0.236+2+3*0.236,-0.236+3*0.236) -- (0.236+2+3*0.236,0.236+3*0.236) -- (0.236+3*0.236,0.236+3*0.236);
		
		\filldraw[blue] (-0.236-3*0.236,-0.236-3*0.236) -- (-0.236+2-3*0.236,-0.236-3*0.236) -- (0.236+2-3*0.236,0.236-3*0.236) -- (0.236-3*0.236,0.236-3*0.236);
		
		\draw[->] (0,-1.5) -- (0,1.5);
		\node[above] at (0,1.5) {$y$};
		\draw[->] (-1.5,0) -- (4,0);
		\node[right] at (4,0) {$x$};
		\node[above] at (4,1.5) {$m=n=2$};
		
		\end{tikzpicture}
		\begin{tikzpicture}
		
		\filldraw[red] (-0.289,-0.167) -- (-0.289+3,-0.167) -- (0.289+3,0.167) -- (0.289,0.167);
		
		\filldraw[blue] (-0.289+3*0.289,-0.167+3*0.167) -- (-0.289+3+3*0.289,-0.167+3*0.167) -- (0.289+3+3*0.289,0.167+3*0.167) -- (0.289+3*0.289,0.167+3*0.167);
		
		\filldraw[blue] (-0.289-3*0.289,-0.167-3*0.167) -- (-0.289+3-3*0.289,-0.167-3*0.167) -- (0.289+3-3*0.289,0.167-3*0.167) -- (0.289-3*0.289,0.167-3*0.167);
		
				\draw[->] (0,-1.5) -- (0,1.5);
		\node[above] at (0,1.5) {$y$};
		\draw[->] (-1.5,0) -- (4,0);
		\node[right] at (4,0) {$x$};
		\node[above] at (4,1.5) {$m=n=3$};
		\end{tikzpicture}

		\caption{The construction in Remark \ref{rem:Variants} (i) for $m= n = 2$ and $m = n = 3$. The red area is the support of $\tilde{\mu}_{m,n}$ and the blue area the supports of $\tilde{\nu}_{m,n}$.}
	\end{center}
\end{figure}

\normalem

%

\end{document}